\documentclass[11pt]{scrartcl}
\usepackage[utf8]{inputenc}
\RequirePackage[T1]{fontenc}

\usepackage{amssymb}
\usepackage{microtype}
\usepackage{amsmath}
\usepackage{amssymb}
\usepackage{amsfonts}
\usepackage{mathtools}
\usepackage[hyphens]{url}
\usepackage{amsthm}
\usepackage{thmtools}
\usepackage[mathscr]{euscript}
\usepackage{thm-restate}
\usepackage{dsfont}

\usepackage{tikz}
\usetikzlibrary{calc}
\usetikzlibrary{hobby}
\usepackage{xcolor}
\usepackage{subcaption}
\pgfdeclarelayer{background}
\pgfdeclarelayer{foreground}
\pgfsetlayers{background,main,foreground}

\usepackage{stmaryrd}
\usepackage{hyperref}
\usepackage{cleveref}

\usepackage{nicefrac}
\usepackage{textpos}

\usepackage{graphicx}
\usepackage{titling}
\usepackage{xspace}
\usepackage{macros}

\usepackage{enumitem}

\hypersetup{
colorlinks=true,
linkcolor=AO!65!black,
citecolor=AO!65!black,
urlcolor=AppleGreen!65!black,
bookmarksopen=true,
bookmarksnumbered,
bookmarksopenlevel=2,
bookmarksdepth=3
}

\title{Counterexample to Babai's lonely colour conjecture}
\predate{}
\date{}
\postdate{}

\preauthor{}
\DeclareRobustCommand{\authorthing}{
    \begin{center}
        James Davies\thanks{Trinity Hall, University of Cambridge, United Kingdom, 
        \href{mailto:jgd37@cam.ac.uk}{jgd37@cam.ac.uk}} \hspace{1cm}
	    Meike Hatzel\thanks{Discrete Mathematics Group, Institute for Basic Science (IBS), Daejeon, South Korea, \href{mailto:research@meikehatzel.com}{research@meikehatzel.com}. Research was supported by the Federal Ministry of Education and
        Research (BMBF) and by a fellowship within the IFI programme of the German Academic Exchange Service (DAAD) and by the Institute for Basic Science (IBS-R029-C1).} \hspace{1cm}
	    Liana Yepremyan\thanks{Department of Mathematics, Emory University, \href{mailto: lyeprem@emory.edu}{lyeprem@emory.edu}.  Research is supported by the National Science Foundation grant 2247013: Forbidden and Colored Subgraphs.}
\end{center}}
\author{\authorthing}
\postauthor{}

\newif\ifcomment
\commentfalse
\commenttrue

\begin{document}

\maketitle

\begin{abstract}
    Motivated by colouring minimal Cayley graphs, in 1978, Babai conjectured that no-lonely-colour graphs have bounded chromatic number.
    We disprove this in a strong sense by constructing graphs of arbitrarily large girth and chromatic number that have a proper edge-colouring in which each cycle contains no colour exactly once.
\end{abstract}

\begin{textblock}{20}(11.5,6)
   \includegraphics[width=80px]{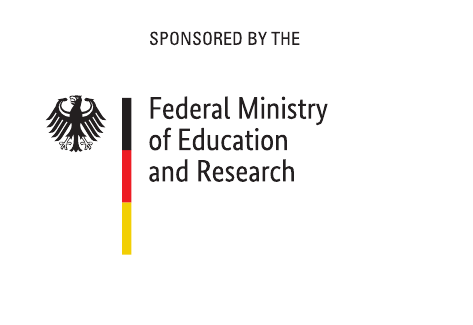}%
\end{textblock}

\section{Introduction}

A graph $G$ is a \emph{\nlcg} if there is an edge-colouring of $G$ such that
\begin{itemize}
    \item every vertex is incident with at most two edges of any colour, and
    \item each cycle contains no colour exactly once.
\end{itemize}

In 1978, Babai~\cite{babai1978chromatic} conjectured that no-lonely-colour graphs have bounded chromatic number.
Very recently, Garcia-Marco and Knauer~\cite{garcia2024colouring} made the weaker conjecture that graphs admitting a proper edge-colouring in which each cycle contains no colour exactly once have bounded chromatic number.
We disprove both of these conjectures in a strong sense; our graphs can also have arbitrarily large girth.

\begin{restatable}{theorem}{constructNoLonelyColourGraphsTheorem}
    \label{thm:no_lonely_colour_graphs_large_chromatic_number}
    For every pair of positive integers $g,k$, there is a graph $G_{g,k}$ with girth at least $g$ and chromatic number at least $k$ that has a proper edge-colouring in which each cycle contains no colour exactly once.
\end{restatable}

We remark that in 1995, Babai~\cite{babai1995automorphism} also conjectured the opposite: the family of no-lonely-colour graphs has unbounded chromatic number. Thus, \cref{thm:no_lonely_colour_graphs_large_chromatic_number} also proves this conjecture.
\cref{thm:no_lonely_colour_graphs_large_chromatic_number} also strengthens a recent theorem of Garcia-Marco and Knauer~\cite{garcia2024colouring} that there are graphs with large chromatic number that have a proper edge-colouring in which no cycle contains each colour at most once.

Babai's~\cite{babai1978chromatic} conjecture was motivated by colouring minimal Cayley graphs. Given a group $\Gamma$ and $S\subseteq \Gamma$, the (undirected) \emph{Cayley}  graph is defined to be a graph on the vertex set $\Gamma$ and $ab$ is an edge whenever $ab^{-1}$ or $ba^{-1}$ is in $S$.
We say $S$ is a \emph{minimal} generating set for $\Gamma$ if $S$ generates $\Gamma$ and no proper subset of it does, in this case we call $\text{Cay}(\Gamma, S)$ a minimal Cayley graph. 

Aspects of minimal Cayley graphs and properties distinguishing them from general Cayley graphs has been extensively studied (see for instance \cite{babai1978chromatic,babai1995automorphism,farrokhi2016groups,garcia2024colouring,godsil1981connectivity,li2001isomorphisms,miklavivc2020minimal,slupik2009minimal,spencer1983s}).
One can observe that every minimal Cayley graph is a no-lonely-colour graph, and thus Babai's conjecture would have implied the following conjecture which remains open as the graphs we construct are not Cayley graphs.
\begin{restatable}[Babai~\cite{babai1978chromatic}]{conjecture}{Babaiconjecture}
    There is a constant $c$ such that every minimal Cayley graphs has chromatic number at most $c$.
\end{restatable}
Although there are various constructions of Cayley graphs with large chromatic number~\cite{bardestani2018chromatic,davies2023chromatic,davies2023prime,davies2024odd,furstenberg1977ergodic,kamae1978van,lubotzky1988ramanujan,nica2024independence,sarkozy1978difference,tomon2015chromatic} (and bounded clique number), it remains open whether or not there is even a minimal Cayley graph with chromatic number at least 5.
Garcia-Marco and Knauer~\cite{garcia2024colouring} showed that the minimal Cayley graph $\text{Cay}(\mathds{Z}_3 \rtimes \mathds{Z}_7,\{(0,1), (1,0)\})$ has chromatic number 4.
They also proved that any minimal Cayley graph of a (finitely generated) generalised dihedral or nilpotent group has chromatic number at most $3$.

We establish \cref{thm:no_lonely_colour_graphs_large_chromatic_number}  by an intricate explicit construction of graphs of arbitrarily large girth and chromatic number (that are furthermore no-lonely-colour graphs).
Erd\H{o}s~\cite{erdos1959graph} was the first who proved that there exist such graphs using the probabilistic method and Lov{\'a}sz~\cite{lovasz1968chromatic} gave the first explicit construction. At this point many  such constructions are known \cite{alon2016coloring,bucic2023explicit,davies2021solution,Davies21,kostochka1999properties,kvrivz1989ahypergraph,lubotzky1988ramanujan,nevsetvril1979short}, many of which have various other desirable properties. Our construction is a delicate modification of Tutte's~\cite{Descartes47,Descartes54} construction of triangle-free graphs with large chromatic number. Other modifications of Tutte's construction have been used before, for example, by Ne{\v{s}}et{\v{r}}il and R{\"o}dl~\cite{nevsetvril1979short} to give a simple explicit construction of graphs with arbitrarily large girth and chromatic number, and more recently by the first author~\cite{Davies21} who constructed intersection graphs of axis-aligned boxes and of lines in $\mathds{R}^3$ with arbitrarily large girth and chromatic number. 

The key difference is a choice and construction of auxiliary hypergraphs of large girth and chromatic number and with further specific properties (we call these \emph{tranquil} hypergraphs). To obtain a class of tranquil hypergraphs, we observe that certain hypergraphs with large girth and chromatic number arising from Gallai's theorem (or equivalently, the multi-dimensional van der Waerden's theorem) have desirable combinatorial properties that may make them a useful gadget of independent interest (see~\cref{thm:tranquil_hypergraphs}). Another variant of these hypergraphs was recently used by Davies, Keller, Kleist, Smorodinsky, and Walczak~\cite{davies2021solution} to resolve Ringel's~\cite{Ringel59} circle problem.

\newpage

The rest of the paper is organised as follows. In~\cref{sec:notation}, we introduce our notation.
In \cref{sec:tranquilhypergraphs}, we examine certain auxiliary hypergraphs that are needed for our construction. In \cref{sec:NLCG}, we prove~\cref{thm:no_lonely_colour_graphs_large_chromatic_number}.
Finally, in \cref{sec:conclusion}, we discuss some related conjectures.

\section{Notation}
\label{sec:notation}

An edge $e$ of a multigraph $G$ is a \emph{bridge} if $G-e$ has more connected components than $G.$

A \emph{closed walk} of length $\ell$ in a hypergraph $H$ is a sequence of alternating vertices and hyperedges $v_0 F_0 v_1 F_1 \dots F_{\ell-1} v_{0}$ such that $v_i,v_{i+1 } \in F_i$ for all $0 \leq i < \ell$, and $F_i\not=F_{i+1}$ for each $0\le i < \ell$ where the sum $i+1$ is taken modulo $\ell$ here and henceforth wherever appropriate, and $v_1,\ldots , v_{\ell -1}$ are all pairwise distinct.
Note that $F_0,\ldots , F_{\ell -1}$ need not all be distinct.
Similarly, we say that a collection of distinct hyperedges $F_1, \ldots , F_{\ell}$ forms a cycle of length $\ell$ if there exist distinct elements $x_1, \ldots , x_{\ell}$ such that for all $i \in \{1, \ldots , \ell - 1\}$ we have $x_i \in F_i \cap F_{i+1}$.
Note that except the vertices $x_i$'s, the edges $F_i$ do not have to be disjoint. This is more commonly known in the literature as \emph{Berge} cycles. The \emph{girth} of a hypergraph is equal to the length of its shortest Berge cycle. 

The \emph{chromatic number} of a hypergraph is the least amount of colours needed to colour the vertices of the hypergraph such that no edge is monochromatic.

\section{Tranquil hypergraphs}
\label{sec:tranquilhypergraphs}

We call an $r$-uniform hypergraph $H$ \emph{\tranquil} if there is a family of labellings $\lambda = \{\lambda_F\}_{F\in E(H)}$ such that for each hyperedge $F\in E(H)$, $\lambda_F: F \to \{1, \ldots , r\}$ is such that for every closed walk $W = v_0F_0v_1F_1\ldots F_{\ell-1}v_0$ of $H$, the multigraph with vertices $\{1,\ldots,r\}$ and  edges $\{\lambda_{F_i}(v_i)\lambda_{F_i}(v_{i+1}) : 0 \le i < \ell\}$ is bridgeless. We refer to this multigraph by $\Projection{\lambda}{W}$ and call it the \emph{projection graph of $W$ with respect to $\lambda$}. See \cref{fig:tranquil_walk_projection} for a very simple example.

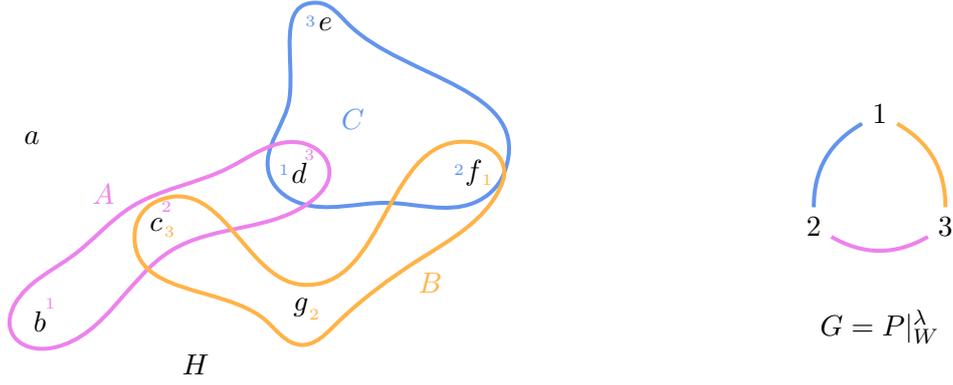
\begin{figure}[!ht]
	\centering
	\begin{tikzpicture}
		\node (hypergraph) at (-4,0) {
			\begin{tikzpicture}
				\node (d) at (0,0) {$d$};
				\node (c) at ($(d)+(200:2)$) {$c$};
				\node (a) at ($(c)+(145:2)$) {$a$};
				\node (b) at ($(c)+(220:2)$) {$b$};
				\node (g) at ($(c)+(330:2.2)$) {$g$};
				\node (f) at ($(d)+(0:2.3)$) {$f$};
				\node (e) at ($(d)+(80:2)$) {$e$};
				
				\draw[line width=1.5pt,CornflowerBlue] ($(e)+(110:0.3)$) to[closed,curve through ={ ($(e)!0.5!(d)+(180:0.3)$) ($(d)+(180:0.4)$) ($(d)+(270:0.4)$) ($(d)!0.5!(f)+(270:0.4)$) ($(f)+(270:0.4)$) ($(f)+(0:0.4)$)}] ($(e)+(0:0.3)$);
				\node[CornflowerBlue] (C-label) at ($(d)+(45:1)$) {$C$};

				\draw[line width=1.5pt,LavenderMagenta] ($(b)+(145:0.4)$) to[closed,curve through ={($(b)!0.5!(c)+(145:0.4)$) ($(c)+(145:0.4)$) ($(c)!0.5!(d)+(90:0.4)$) ($(d)+(90:0.4)$) ($(d)+(0:0.4)$) ($(d)+(300:0.4)$) ($(c)+(300:0.4)$) ($(b)+(300:0.4)$)}] ($(b)+(200:0.4)$);
				\node[LavenderMagenta] (A-label) at ($(c)+(150:0.8)$) {$A$};
				
				\draw[line width=1.5pt,PastelOrange] ($(c)+(90:0.3)$) to[closed,curve through ={($(g)+(90:0.3)$) ($(f)+(90:0.4)$) ($(f)+(0:0.4)$) ($(f)+(330:0.4)$) ($(f)!0.5!(g)+(330:0.5)$) ($(g)+(330:0.5)$) ($(g)+(270:0.5)$) ($(g)+(200:0.5)$)}] ($(c)+(200:0.3)$);
				\node[PastelOrange] (B-label) at ($(f)!0.5!(g)+(315:0.8)$) {$B$};
				
				\node[LavenderMagenta] at ($(b)+(60:0.27)$) {\tiny{$1$}};
				\node[LavenderMagenta] at ($(c)+(60:0.27)$) {\tiny{$2$}};
				\node[LavenderMagenta] at ($(d)+(60:0.27)$) {\tiny{$3$}};
				
				\node[PastelOrange] at ($(f)+(330:0.2)$) {\tiny{$1$}};
				\node[PastelOrange] at ($(g)+(330:0.2)$) {\tiny{$2$}};
				\node[PastelOrange] at ($(c)+(330:0.2)$) {\tiny{$3$}};

				\node[CornflowerBlue] at ($(d)+(170:0.2)$) {\tiny{$1$}};
				\node[CornflowerBlue] at ($(f)+(170:0.2)$) {\tiny{$2$}};
				\node[CornflowerBlue] at ($(e)+(170:0.2)$) {\tiny{$3$}};
		\end{tikzpicture}};
		
		\node (multigraph) at (4,0) {
			\begin{tikzpicture}
				\node (centre) at (0,0) {};
				\node (1) at ($(centre)+(90:1)$) {$1$};
				\node (2) at ($(centre)+(210:1)$) {$2$};
				\node (3) at ($(centre)+(330:1)$) {$3$};
				
				\draw[line width=1.5pt,PastelOrange,bend left] (1) to (3);
				\draw[line width=1.5pt,LavenderMagenta,bend left] (3) to (2);
				\draw[line width=1.5pt,CornflowerBlue,bend left] (2) to (1);
		\end{tikzpicture}};

        \node (H) at (-5,-2.5) {$H$};
        \node (G) at (4,-2) {$G=\Projection{\lambda}{W}$};
	\end{tikzpicture}
	\caption{On the left, we see a 3-uniform hypergraph $H$ with three edges and a labelling for each edge.
		The walk $W=dAcBfCd$ is the unique closed walk in $H$.
		On the right, we see the multigraph $G$ corresponding to $W,$ which is bridgeless.
		Thus, the hypergraph $H$ is tranquil.}
	\label{fig:tranquil_walk_projection}
\end{figure}

The aim of this section is to construct tranquil hypergraphs with large girth and chromatic number.
These are crucial gadgets in our construction of \nlcgs with large girth and chromatic number.
Our construction of tranquil hypergraphs comes from the multi-dimensional van der Waerden's theorem, or equivalently, Gallai's theorem (see, for example, \cite{Rado45}). 

\begin{theorem}\label{thm:Gallai}
    For every $k, d \geq 1$, and for every finite set $T\subseteq \mathds{Z}^d$ and every colouring of $\mathds{Z}^d$ by $k$ colours, there exist $a\in \mathds{Z}^d$ and $r>0$ such taht the set $\{a+rt:t\in T\}$ is monochromatic.
\end{theorem}

For each integer $d\ge 2$, we let $C_d$ be the $d$-uniform hypergraph on the vertex set $\mathds{Z}^d$ with hyperedges $E(C_d)=\{ \{x+re_1, \ldots , x+re_d\} :x\in \mathds{Z}^d, r\in \mathds{N} \}$, where $e_1,\ldots , e_d$ is the standard basis of $\mathds{Z}^d$.
It follows from \cref{thm:Gallai} that $C_d$ has unbounded chromatic number.
Next, we show that $C_d$ is tranquil for each $d\ge 2$.

\begin{lemma}\label{lem:tranquil}
    For each integer $d\ge 2$, the hypergraph $C_d$ is tranquil.
\end{lemma}

\begin{proof}
    For each hyperedge $F=\{x+re_1, \ldots , x+re_d\}$ of $C_d$, let $\lambda_F: F \to \{1, \ldots , d\}$ be the labelling such that $\lambda_F(x+re_i)= i$ for each $1\le i \le d$.
    Let $v_0F_0v_1F_1\ldots F_{n-1}v_0$ be a closed walk of $C_d$.
    Without loss of generality, we may assume that $v_0=\mathbf{0}$.
    Let $D$ be the (directed) multigraph with $V(D)=\{1,\ldots,d\}$ and $E(D)=\{\lambda_{F_i}(v_i)\lambda_{F_i}(v_{i+1}) : 0\le i < n\}$.
    Then for each $F_i$, there exists a positive integer $r_{i}$ such that $v_{i+1}=v_i + r_{i}( e_{ \lambda_{F_i}(v_{i+1}) } - e_{ \lambda_{F_i}(v_{i}) } )$.
    In particular, for each edge $xy$ of $D$, there is a positive integer $r_{xy}$, so that
    \[
    \sum_{xy\in E(D)} r_{xy}(e_y-e_x)=0.
    \]

    For each $z\in V(D)$, let $E_z^{\variablestyle{in}}$ be the edges $xz$ of $D$ and let $E_z^\variablestyle{out}$ be the edges $zy$ of $D$.
    Then for each $z\in V(D)$, we have that
    \[
    \sum_{e\in E_z^\variablestyle{in}} r_{e}-\sum_{e\in E_z^\variablestyle{out}} r_{e}=0.
    \]

    Suppose for sake of contradiction that $D$ contains a bridge $ab$, and let $A,B$ be a partition of $V(D)$ with $a\in A$, $b\in B$, and no edge between $A$ and $B$ in $D-ab$.
    Then,
    \[
    r_{ab}
    =
    \sum_{e\in E_a^{\variablestyle{in}}} r_{e}-\sum_{e\in E_a^{\variablestyle{out}}\backslash \{ab\}} r_{e}= \sum_{z\in A\backslash \{a\}} 
    -
    \left( 
    \sum_{e\in E_z^{\variablestyle{in}}} r_{e}-\sum_{e\in E_z^{\variablestyle{out}}} r_e
    \right)
    =
    0,
    \]
    which contradicts the fact that $r_{ab}$ is a positive integer.
    Therefore, $D$ is bridgeless, and thus $C_d$ is tranquil, as desired.
\end{proof}

Although $C_d$ does not have large girth, Prömel and Voigt~\cite{PV88,PV90} proved that it does have induced subhypergraphs with large girth and still large chromatic number. The following is one formulation of a special case of Prömel and Voigt's~\cite{PV88,PV90} sparse Gallai's theorem.

\begin{theorem}[Prömel, Voigt~\cite{PV88,PV90}]\label{Gallai}
    For every triple of positive integers $d,k,g$ with $d\ge 3$, there is a (finite) induced subhypergraph $H_{d,k,g}$ of $C_d$ with chromatic number at least $k$ and girth at least $g$.
\end{theorem}

Prömel and Voigt's~\cite{PV88,PV90} construction of such (finite) hypergraphs is explicit.
As the property of being tranquil is preserved by removing vertices, we obtain the following from \cref{lem:tranquil,Gallai}.

\begin{theorem}
    \label{thm:tranquil_hypergraphs}
    For every triple of positive integers $r,k,g$ with $r\ge 3$, there exists a \tranquil $r$-uniform hypergraph $H_{g,k,r}$ with chromatic number at least $k$ and girth at least $g$.
\end{theorem}


\section{\NLCGs}
\label{sec:NLCG}

In this section, we use the hypergraphs $H_{g,k,r}$ from~\cref{thm:tranquil_hypergraphs} to construct graphs $G_{g,k}$ proving~\cref{thm:no_lonely_colour_graphs_large_chromatic_number}.

\paragraph{Hyper-Tutte-construction:}
We use a modification of Tutte's construction~\cite{Descartes47,Descartes54} to inductively construct a family of graphs $G_{g,k}$ as follows, see~\cref{fig:hyper_tutte} for an illustration.
Let $G_{g,1}$ be a single vertex.
Assuming that $G_{g,k-1}$ has been constructed we construct $G_{g,k}$ as follows.
Let $r \coloneqq \Abs{G_{g,k-1}}$, in particular we assume $V(G_{g,k-1})=\{1, 2, \dots, r\}$.
We use the hypergraph $H=H_{\lceil g/3 \rceil ,k,r}$ given by~\cref{thm:tranquil_hypergraphs}.
As $H$ is \tranquil, there is a collection of labellings $\lambda = \{\lambda_F: F \to \{1, \ldots , r\} : F \in \E{H}\}$ such that 
\begin{enumerate}[label=\textbf{(TR)},leftmargin=5em]
    \item \label{item:tranq} for every closed walk $W$ in $H$, the multigraph $\Projection{\lambda}{W}$ is bridgeless.
\end{enumerate}
We obtain $G_{g,k}$ by taking the vertex set to be $\Abs{\E{H}}$ many copies of $G_{g,k-1}$, let's call this collection $\mathcal{G}$, together with an independent set $T$ of size $\Abs{\V{H}}$, and defining edges as follows.
We consider two bijections; $\nu : T \rightarrow \V{H}$ and $\mu : \E{H} \rightarrow \mathcal{G}.$
Now for every edge $F\in E(H)$ we add a perfect matching between $\nu^{-1}(F)$ and $\mu(F)$ such that every vertex $i \in \mu(F)$ is mapped to the vertex $v\in \nu^{-1}(F)$ with $\lambda_F(\nu(v))=i$.

\tikzstyle{vertex} = [draw, circle, inner sep=.3mm, thick,fill=black]

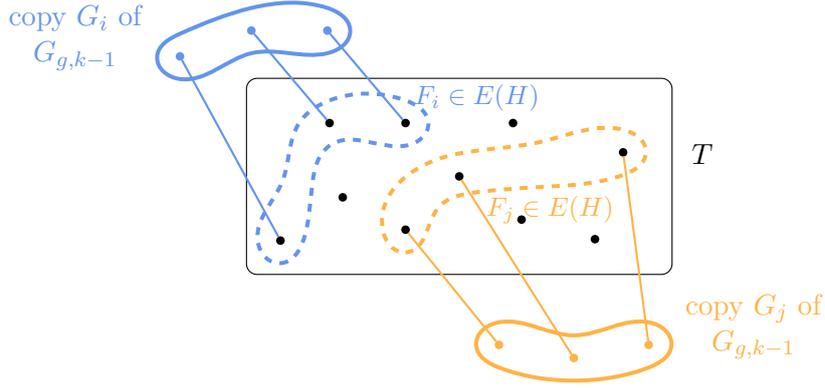
\begin{figure}[!ht]
    \centering
    \begin{tikzpicture}
        \node (center) at (0,0) {};

        \node[vertex] (x-1) at ($(center)$) {};
        \node[vertex] (x-2) at ($(x-1)+(135:1)$) {};
        \node[vertex] (x-3) at ($(x-2)+(180:1)$) {};
        \node[vertex] (x-4) at ($(x-3)+(280:1)$) {};
        \node[vertex] (x-5) at ($(x-4)+(215:1)$) {};
        \node[vertex] (x-6) at ($(x-1)+(225:1)$) {};
        \node[vertex] (x-7) at ($(x-1)+(45:1)$) {};
        \node[vertex] (x-8) at ($(x-1)+(325:1)$) {};
        \node[vertex] (x-9) at ($(x-8)+(345:1)$) {};
        \node[vertex] (x-10) at ($(x-7)+(345:1.5)$) {};

        \draw[rounded corners] ($(center)-(2.8,1.3)$) rectangle ($(center)+(2.8,1.3)$);

        \node (T) at ($(center)+(3.2,0.3)$) {$T$};

        \draw[line width=1.5pt,CornflowerBlue,dashed] ($(x-5)+(270:0.3)$) to[closed,curve through ={ ($(x-5)+(0:0.3)$) ($(x-5)!0.5!(x-3)+(0:0.1)$) ($(x-3)+(270:0.3)$) ($(x-2)+(270:0.3)$) ($(x-2)+(0:0.3)$) ($(x-2)+(90:0.3)$) ($(x-3)+(90:0.3)$) ($(x-3)+(180:0.4)$) ($(x-5)!0.5!(x-3)+(180:0.3)$)}] ($(x-5)+(180:0.3)$);

        \draw[line width=1.5pt,PastelOrange,dashed] ($(x-6)+(270:0.3)$) to[closed,curve through ={ ($(x-6)!0.5!(x-1)+(315:0.1)$) ($(x-1)+(315:0.3)$) ($(x-1)!0.5!(x-10)+(270:0.3)$) ($(x-10)+(270:0.3)$) ($(x-10)+(0:0.3)$) ($(x-10)+(90:0.3)$) ($(x-10)!0.5!(x-1)+(90:0.3)$) ($(x-1)+(110:0.3)$) }] ($(x-6)+(180:0.3)$);

        \node[vertex,PastelOrange] (j-2) at ($(x-9)+(260:1.6)$) {};
        \node[vertex,PastelOrange] (j-1) at ($(j-2)+(170:1)$) {};
        \node[vertex,PastelOrange] (j-3) at ($(j-2)+(10:1)$) {};

        \draw[line width=1.5pt,PastelOrange] ($(j-1)+(270:0.3)$) to[closed,curve through ={ ($(j-2)+(270:0.3)$) ($(j-3)+(270:0.3)$) ($(j-3)+(0:0.3)$) ($(j-3)+(90:0.3)$) ($(j-2)+(90:0.3)$) ($(j-1)+(90:0.3)$)}] ($(j-1)+(180:0.3)$);

        \draw[line width=0.8pt,PastelOrange] (j-1) to (x-6);
        \draw[line width=0.8pt,PastelOrange] (j-2) to (x-1);
        \draw[line width=0.8pt,PastelOrange] (j-3) to (x-10);

        \node[PastelOrange,align=center] (G-j) at ($(j-3)+(10:1.4)$) {copy $G_j$ of\\$G_{g,k-1}$};
        \node[PastelOrange,align=center] (F-j) at ($(x-8)+(20:0.4)$) {\small{$F_j \in \E{H}$}};

        \node[vertex,CornflowerBlue] (i-2) at ($(x-3)+(130:1.6)$) {};
        \node[vertex,CornflowerBlue] (i-1) at ($(i-2)+(200:1)$) {};
        \node[vertex,CornflowerBlue] (i-3) at ($(i-2)+(360:1)$) {};

        \draw[line width=1.5pt,CornflowerBlue] ($(i-1)+(270:0.3)$) to[closed,curve through ={ ($(i-2)+(270:0.3)$) ($(i-3)+(270:0.3)$) ($(i-3)+(0:0.3)$) ($(i-3)+(90:0.3)$) ($(i-2)+(90:0.3)$) ($(i-1)+(90:0.3)$)}] ($(i-1)+(180:0.3)$);

        \draw[line width=0.8pt,CornflowerBlue] (i-1) to (x-5);
        \draw[line width=0.8pt,CornflowerBlue] (i-2) to (x-3);
        \draw[line width=0.8pt,CornflowerBlue] (i-3) to (x-2);

        \node[CornflowerBlue,align=center] (G-i) at ($(i-1)+(170:1.4)$) {copy $G_i$ of\\$G_{g,k-1}$};
        \node[CornflowerBlue,align=center] (F-i) at ($(x-2)+(20:1)$) {\small{$F_i \in \E{H}$}};
    \end{tikzpicture}
    \caption{Additional to the independent set $T$, for every edge $F_i \in \E{H}$ we add a copy $G_i$ of $G_{g,k-1}$ and a matching between it and the corresponding vertices in $T.$}
    \label{fig:hyper_tutte}
\end{figure}

\paragraph{Cycles in the constructed graphs:}
Let $C$ be a cycle in $G_{g,k}$ that does not lie completely in a copy of $G_{g,k-1}$.
We describe how $C$ induces a closed walk in $H$, which we refer to as $\hyperwalk{H}{C}$.
Consider $G_1,G_2,\dots,G_{\ell}$ to be the copies of $G_{g,k-1}$ the cycle $C$ passes through in that order (note that $G_i$'s do not have to be distinct).
Let $P_i$ be the segment of $C$ lying in $G_i$ starting in vertex $\Start{P_i}$ and ending in vertex $\End{P_i}$. 

Then $\hyperwalk{H}{C}$ is the following sequence: 

\begin{gather*}
    \Fkt{\lambda^{-1}_{\Fkt{\mu^{-1}}{G_1}}}{\Start{P_1}},\Fkt{\mu^{-1}}{G_1},
    \Fkt{\lambda^{-1}_{\Fkt{\mu^{-1}}{G_1}}}{\End{P_1}},
    \Fkt{\mu^{-1}}{P_2},\\
    \Fkt{\lambda^{-1}_{\Fkt{\mu^{-1}}{P_2}}}{\End{P_2}},
    \dots,
    \Fkt{\mu^{-1}}{G_{\ell}},
    \Fkt{\lambda^{-1}_{\Fkt{\mu^{-1}}{G_{\ell}}}}{\End{P_{\ell}}}.
\end{gather*}

\begin{figure}[!ht]
    \centering
    \begin{tikzpicture}
        \node (centre) at (0,0) {};

        \node[fill=white,circle] (a) at ($(centre)+(120:2)$) {$a$};
        \node[fill=white,circle] (c) at ($(centre)+(240:2)$) {$c$};
        \node[fill=white,circle] (d) at ($(centre)+(300:2)$) {$d$};
        \node[fill=white,circle] (f) at ($(centre)+(60:2)$) {$f$};
        \node[fill=white,circle] (b) at ($(a)!0.5!(c)+(180:2)$) {$b$};
        \node[fill=white,circle] (e) at ($(f)!0.5!(d)+(0:2)$) {$e$};

        \draw[rounded corners] ($(center)-(3.5,2.4)$) rectangle ($(center)+(3.5,2.4)$);
        \node[vertex] (ab-1) at ($(a)!0.33!(b)$) {};
        \node[vertex] (ab-1) at ($(a)!0.66!(b)$) {};

        \node[vertex] (ef-1) at ($(e)!0.33!(f)$) {};
        \node[vertex] (ef-1) at ($(e)!0.66!(f)$) {};

        \node[vertex] (bc-1) at ($(b)!0.33!(c)$) {};
        \node[vertex] (bc-1) at ($(b)!0.66!(c)$) {};

        \node[vertex] (de-1) at ($(d)!0.33!(e)$) {};
        \node[vertex] (de-1) at ($(d)!0.66!(e)$) {};

        \node[vertex] at ($(a)!0.6!(b)+(315:0.7)$) {};
        \node[vertex] at ($(c)!0.5!(b)+(45:0.7)$) {};

        \node[vertex] at ($(f)!0.7!(e)+(225:0.7)$) {};
        \node[vertex] at ($(d)!0.55!(e)+(135:0.7)$) {};

        \node[vertex] at ($(a)!0.4!(b)+(135:0.8)$) {};
        \node[vertex] at ($(a)!0.8!(b)+(135:0.7)$) {};
        \node[vertex] at ($(a)!0.63!(b)+(135:1.3)$) {};
        \node[vertex] at ($(a)!0.2!(b)+(135:1)$) {};

        \node[vertex] at ($(f)!0.4!(e)+(45:1.1)$) {};
        \node[vertex] at ($(f)!0.8!(e)+(45:0.7)$) {};
        \node[vertex] at ($(f)!0.63!(e)+(45:1.3)$) {};
        \node[vertex] at ($(f)!0.2!(e)+(45:0.7)$) {};

        \node[vertex] at ($(b)!0.4!(c)+(225:1.1)$) {};
        \node[vertex] at ($(b)!0.8!(c)+(225:1)$) {};
        \node[vertex] at ($(b)!0.63!(c)+(225:0.5)$) {};
        \node[vertex] at ($(b)!0.2!(c)+(225:0.7)$) {};

        \node[vertex] at ($(e)!0.4!(d)+(315:0.8)$) {};
        \node[vertex] at ($(e)!0.8!(d)+(315:0.7)$) {};
        \node[vertex] at ($(e)!0.63!(d)+(315:1.3)$) {};
        \node[vertex] at ($(e)!0.2!(d)+(315:1)$) {};

        \node[CornflowerBlue] (F-1) at ($(a)!0.6!(b)+(135:0.7)$) {$F_1$};
        \node[PastelOrange] (F-5) at ($(f)!0.6!(e)+(45:0.7)$) {$F_5$};
        \node[PastelOrange] (F-2) at ($(b)!0.4!(c)+(225:0.7)$) {$F_2$};
        \node[CornflowerBlue] (F-4) at ($(e)!0.63!(d)+(315:0.7)$) {$F_4$};
        \node[LavenderMagenta] (F-3) at ($(centre)+(150:0.4)$) {$F_3$};

        \draw[line width=1.5pt,CornflowerBlue,dashed] ($(a)+(135:0.3)$) to[closed,curve through ={($(a)+(45:0.3)$) ($(a)+(315:0.3)$) ($(a)!0.5!(b)+(315:0.3)$) ($(b)+(315:0.3)$) ($(b)+(225:0.3)$) ($(b)+(135:0.3)$)}] ($(b)!0.5!(a)+(135:0.3)$);
        \node[vertex] (ap-1) at ($(a)+(135:3)$) {};
        \node[vertex] (bp-1) at ($(b)+(125:3.8)$) {};
        \def\dist{0.5}
        \draw[line width=1.5pt,CornflowerBlue] ($(ap-1)+(135:\dist)$) to[closed,curve through ={($(ap-1)+(45:\dist)$) ($(ap-1)+(315:\dist)$) ($(ap-1)!0.5!(bp-1)+(315:\dist)$) ($(bp-1)+(315:\dist)$) ($(bp-1)+(225:\dist)$) ($(bp-1)+(135:\dist)$)}] ($(bp-1)!0.5!(ap-1)+(135:\dist)$);

        \draw[line width=1.5pt,CornflowerBlue,dashed] ($(e)+(135:0.3)$) to[closed,curve through ={($(e)+(45:0.3)$) ($(e)+(315:0.3)$) ($(e)!0.5!(d)+(315:0.3)$) ($(d)+(315:0.3)$) ($(d)+(225:0.3)$) ($(d)+(135:0.3)$)}] ($(d)!0.5!(e)+(135:0.3)$);
        \node[vertex] (ep-4) at ($(e)+(315:3.8)$) {};
        \node[vertex] (dp-4) at ($(d)+(315:3)$) {};
        \draw[line width=1.5pt,CornflowerBlue] ($(ep-4)+(135:\dist)$) to[closed,curve through ={($(ep-4)+(45:\dist)$) ($(ep-4)+(315:\dist)$) ($(ep-4)!0.5!(dp-4)+(315:\dist)$) ($(dp-4)+(315:\dist)$) ($(dp-4)+(225:\dist)$) ($(dp-4)+(135:\dist)$)}] ($(dp-4)!0.5!(ep-4)+(135:\dist)$);

        \draw[line width=1.5pt,PastelOrange,dashed] ($(f)+(225:0.3)$) to[closed,curve through ={($(f)+(135:0.3)$) ($(f)+(45:0.3)$) ($(f)!0.5!(e)+(45:0.3)$) ($(e)+(45:0.3)$) ($(e)+(315:0.3)$) ($(e)+(225:0.3)$)}] ($(e)!0.5!(f)+(225:0.3)$);
        \node[vertex] (fp-5) at ($(f)+(35:3.8)$) {};
        \node[vertex] (ep-5) at ($(e)+(35:3)$) {};
        \draw[line width=1.5pt,PastelOrange] ($(fp-5)+(225:\dist)$) to[closed,curve through ={($(fp-5)+(135:\dist)$) ($(fp-5)+(45:\dist)$) ($(fp-5)!0.5!(ep-5)+(45:\dist)$) ($(ep-5)+(45:\dist)$) ($(ep-5)+(315:\dist)$) ($(ep-5)+(225:\dist)$)}] ($(ep-5)!0.5!(fp-5)+(225:\dist)$);

        \draw[line width=1.5pt,PastelOrange,dashed] ($(b)+(225:0.3)$) to[closed,curve through ={($(b)+(135:0.3)$) ($(b)+(45:0.3)$) ($(b)!0.5!(c)+(45:0.3)$) ($(c)+(45:0.3)$) ($(c)+(315:0.3)$) ($(c)+(225:0.3)$)}] ($(c)!0.5!(b)+(225:0.3)$);
        \node[vertex] (bp-2) at ($(b)+(225:3.8)$) {};
        \node[vertex] (cp-2) at ($(c)+(225:3)$) {};
        \draw[line width=1.5pt,PastelOrange] ($(bp-2)+(225:\dist)$) to[closed,curve through ={($(bp-2)+(135:\dist)$) ($(bp-2)+(45:\dist)$) ($(bp-2)!0.5!(cp-2)+(45:\dist)$) ($(cp-2)+(45:\dist)$) ($(cp-2)+(315:\dist)$) ($(cp-2)+(225:\dist)$)}] ($(cp-2)!0.5!(bp-2)+(225:\dist)$);

        \draw[line width=1.5pt,LavenderMagenta,dashed] ($(a)+(90:0.3)$) to[closed,curve through ={($(a)!0.5!(f)+(90:0.4)$) ($(f)+(90:0.3)$) ($(f)+(0:0.3)$) ($(f)!0.5!(d)+(0:0.3)$) ($(d)+(0:0.3)$) ($(d)+(270:0.3)$) ($(d)!0.5!(c)+(270:0.4)$) ($(c)+(270:0.3)$) ($(c)+(180:0.3)$) ($(c)!0.5!(a)+(180:0.3)$)}] ($(a)+(180:0.3)$);
        \node[vertex] (ap-3) at ($(a)+(80:3)$) {};
        \node[vertex] (cp-3) at ($(c)+(80:6)$) {};
        \node[vertex] (fp-3) at ($(f)+(80:3)$) {};
        \node[vertex] (dp-3) at ($(d)+(80:6)$) {};
        \draw[line width=1.5pt,LavenderMagenta] ($(ap-3)+(90:\dist)$) to[closed,curve through ={($(ap-3)!0.5!(fp-3)+(90:\dist)$) ($(fp-3)+(90:\dist)$) ($(fp-3)+(0:\dist)$) ($(dp-3)+(0:\dist)$) ($(dp-3)+(270:\dist)$) ($(dp-3)!0.5!(cp-3)+(270:\dist)$) ($(cp-3)+(270:\dist)$) ($(cp-3)+(180:\dist)$)}] ($(ap-3)+(180:\dist)$);

        \begin{pgfonlayer}{background}
            \draw[AppleGreen] (a) to (ap-3);
            \draw[AppleGreen] (a) to (ap-1);
            \draw[AppleGreen] (b) to (bp-1);
            \draw[AppleGreen] (b) to (bp-2);
            \draw[AppleGreen] (c) to (cp-2);
            \draw[AppleGreen] (c) to (cp-3);
            \draw[AppleGreen] (d) to (dp-3);
            \draw[AppleGreen] (d) to (dp-4);
            \draw[AppleGreen] (e) to (ep-4);
            \draw[AppleGreen] (e) to (ep-5);
            \draw[AppleGreen] (f) to (fp-5);
            \draw[AppleGreen] (f) to (fp-3);

            \draw[dotted,AppleGreen,line width=1.5pt] (ap-1) to (bp-1);
            \draw[dotted,AppleGreen,line width=1.5pt] (bp-2) to (cp-2);
            \draw[dotted,AppleGreen,line width=1.5pt] (dp-4) to (ep-4);
            \draw[dotted,AppleGreen,line width=1.5pt] (ep-5) to (fp-5);
            \draw[dotted,AppleGreen,line width=1.5pt] (ap-1) to (bp-1);
            \draw[dotted,AppleGreen,line width=1.5pt] (ap-3) to (fp-3);
            \draw[dotted,AppleGreen,line width=1.5pt] (cp-3) to (dp-3);
        \end{pgfonlayer}
    \end{tikzpicture}
    \caption{The cycle $C$ in $G_{g,k}$ drawn in \textcolor{AppleGreen}{green} corresponds to the walk $\hyperwalk{H}{C} = a F_1 b F_2 c F_3 d F_4 e F_5 f F_3 a$ in the hypergraph $H$.}
    \label{fig:cycles}
\end{figure}
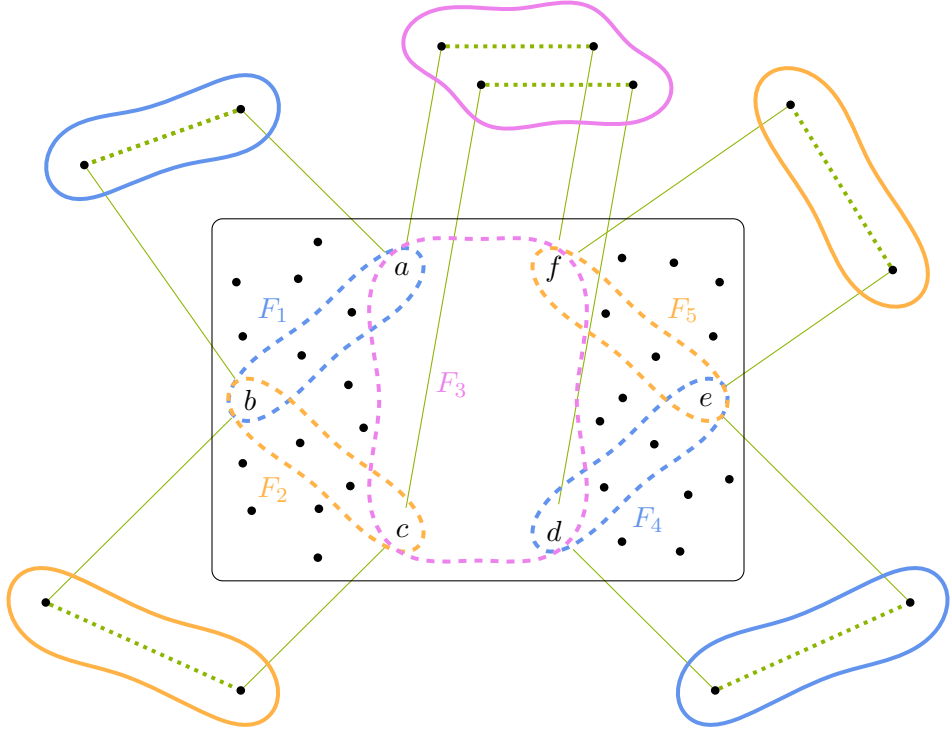


We now show that $G_{g,k}$ has a proper edge-colouring in which each cycle contains no colour exactly once.
In particular, this shows that $G_{g,k}$ is a \nlcg.

\begin{lemma}
    \label{thm:no_lonely_colour}
    For every pair of positive integers $g,k$, the graph $G_{g,k}$ has a proper edge-colouring in which each cycle contains no colour exactly once.
\end{lemma}
\begin{proof}
    Clearly $G_{g,1}$ has such an edge-colouring since $G_{g,1}$ is edgeless.
    So, we argue inductively assuming that $G_{g,k-1}$ has a proper edge-colouring $\kappa_{k-1}: \V{G} \to \N$ such that each cycle contains no colour exactly once.
    Let $H = H_{\lceil g/3 \rceil,k,r}$ be the hypergraph from~\cref{thm:tranquil_hypergraphs} used in the last step of the construction of $G_{g,k}$.
    
    We colour the edges inside of all copies of $G_{g,k-1}$ according to $\kappa_{k-1}$.
    Then, for each $F \in \E{H}$ we colour the matching between $\Fkt{\mu}{F}$ and $T$ with a new colour, thus adding $\Abs{\E{H}}$ colours, to obtain the edge-colouring $\kappa_{k}$ of $G_{g,k}$.

    The edge-colouring $\kappa_k$ is proper as $\kappa_{k-1}$ is a proper edge-colouring and each new colour induces a matching between $T$ and some $\Fkt{\mu}{F}$.

    It remains to show that each cycle $C$ of $G_{g,k}$ contains no colour exactly once.
    By induction, if $C$ completely lies in some copy of $G_{g,k-1}$ in $G_{g,k}$, then $C$ contains no colour exactly once. So, we may assume otherwise.
    Let $C$ be such a cycle, and $\hyperwalk{H}{C}$ be the walk it induces in $H$.
    Let $\lambda = \{\lambda_F: F \to \{1, \ldots , r\} : F \in \E{H}\}$ be the collection of labellings witnessing that $H$ is \tranquil, i.e.~that fulfills~\cref{item:tranq} that is used in the construction of $G_{g,k}$.

    As before let $G_1,\dots, G_{\ell}$ be the copies of $G_{g,k-1}$ that the cycle $C$ passes through in that order and let $P_i$ be the maximal subpath of $C$ in $G_i$.

    Now suppose towards a contradiction that $C$ contains an edge $e$ such that $\Fkt{\kappa_k}{e} \neq \Fkt{\kappa_k}{e'}$ for all edges $e'$ in $C$ with $e \neq e'$.
    As $C$ contains an even number of edges from every added matching, the edge $e$ has to lie in some $C_i.$
    Consider the corresponding hyperedge $F_i$ and vertices $x,y \in \V{H}$ such that $\Fkt{\lambda_{F_i}}{x} = \Start{C_i}$ and $\Fkt{\lambda_{F_i}}{y} = \End{C_i}$.
    As $\Projection{\lambda}{\hyperwalk{H}{C}}$ is bridgeless, the edge $xy$ is not a bridge and there is an $x$-$y$-path $Q$ in $\Projection{\lambda}{\hyperwalk{H}{C}} - xy,$ 
    this corresponds to a $x$-$y$-walk in $H$, say using the edges $F_{i_1}, \dots, F_{i_h}$.
    We consider the Eulerian multigraph $\tilde{G}$ on the vertex set of $G_{g,k-1}$ obtained by adding disjoint copies of the paths $P_{i_1},\dots, P_{i_h},$ and $P_i$, that is, an edge lying in multiple of these paths is added multiple times.
    The underlying simple graph of $\tilde{G}$ is a subgraph of $G_{g,k-1}$.
    If $e$ has a parallel multi-edge in $\tilde{G}$, then this edge corresponds to an edge in $C$ having the same colour as $e$, a contradiction.
    So $e$ lies on a cycle of length at least three in $\tilde{G}$.
    This cycle corresponds to a cycle in $G_{g,k-1}$ and as such, by induction hypothesis, contains another edge $e'$ with $\Fkt{\kappa_{k}}{e'} = \Fkt{\kappa_{k-1}}{e'} = \Fkt{\kappa_{k-1}}{e} = \Fkt{\kappa_{k}}{e}$, a contradiction.
    Thus no cycle in $G_{g,k}$ contains a colour of $\kappa_{k}$ exactly once.
\end{proof}


Next, we routinely show that $G_{g,k}$ has large girth and chromatic number.

\begin{lemma}
    \label{lem:girth}
    For every pair of positive integers $g,k$, the graph $G_{g,k}$ has girth at least $g$.
\end{lemma}
\begin{proof}
    Clearly, the lemma holds for $k=1$ since $G_{g,1}$ is the single vertex graph.
    So, we proceed inductively on $k$ assuming that $G_{g,k-1}$ has girth at least $g$.
    Let $H = H_{\lceil g/3 \rceil,k,r}$ be the hypergraph from~\cref{thm:tranquil_hypergraphs} used in the last step of the construction of $G_{g,k}$.
    Note in particular that $H$ has girth at least $\Ceil{\nicefrac{g}{3}}$.
    Now let $C$ be a shortest cycle in $G_{g,k}$.
    If $C$ completely lies in a copy of $G_{g,k-1},$ then it has length at least $g$ by induction, so we may assume otherwise.
    Then $C$ induces the walk $\hyperwalk{H}{C}$ in $H$, which has length at least $\Ceil{\nicefrac{g}{3}}$, because $H$ has girth at least $\Ceil{\nicefrac{g}{3}}$.
    As the copies of $G_{g,k-1}$ are connected to $T$ by matchings, each edge in $\hyperwalk{H}{C}$ corresponds to a subpath of $C$ of length at least three (two matching edges and at least one edge in a copy of $G_{g,k}$).
    Thus $C$ is of length at least $3\cdot\Ceil{\nicefrac{g}{3}} \ge g$.
\end{proof}

\begin{lemma}
    \label{lem:chromatic_number}
    For every pair of positive integers $g,k$, we have $\chrom{G_{g,k}} \geq k$.
\end{lemma}
\begin{proof}
    Clearly, the lemma holds for $k=1$.
    So, we proceed inductively on $k$ assuming that $\chi(G_{g,k-1})\ge k-1$.
    Suppose towards a contradiction that there was a proper colouring of $G_{g,k}$ with $k-1$ colours.
    Let again $H = H_{\lceil g/3 \rceil,k,r}$ be the hypergraph from~\cref{thm:tranquil_hypergraphs} used in the last step of the construction of $G_{g,k}$.
    As $H$ has chromatic number at least $k,$ there is a subset $I \subseteq T$ such that $F=\Fkt{\nu}{I}$ is an edge in $H$, and $F$ is monochromatic.
    As every vertex in $\Fkt{\mu}{F}$ has a neighbour in $F$ this colour cannot be used in $\Fkt{\mu}{F}$ and thus $\Fkt{\mu}{F}$ is coloured with $k-2$ colours, a contradiction.
\end{proof}

Together \cref{thm:no_lonely_colour}, \cref{lem:girth}, and \cref{lem:chromatic_number} imply our main result~\cref{thm:no_lonely_colour_graphs_large_chromatic_number}.

\section{Concluding Remarks}
\label{sec:conclusion}

We conclude with two open problems. First, we believe the following natural extension of \cref{thm:no_lonely_colour_graphs_large_chromatic_number} should be true as well.

\begin{conjecture}
    For every pair of positive integers $m,k$, there exists a graph with chromatic number at least $k$ that has a proper edge-colouring such that each cycle contains no colour between one and $m$ times.
\end{conjecture}

Secondly, we would like to restate Babai's original conjecture on minimal Cayley graphs and, in particular, ask whether our construction for \cref{thm:no_lonely_colour_graphs_large_chromatic_number} can be modified to obtain a minimal Cayley graph disproving it.

\Babaiconjecture*

We remark that Babai also made the same conjecture for semi-minimal Cayley graphs.
Given a group $\Gamma$ and $S\subseteq \Gamma$ with $S=\{s_1, \ldots , s_n\}$, we say that $S$ is a \emph{semi-minimal} generating set for $\Gamma$ if $S$ generates $\Gamma$ and no $s_i$ is expressible in terms of $s_1, \ldots , s_{i-1}$, in this case we call $\text{Cay}(\Gamma, S)$ a semi-minimal Cayley graph.

\begin{conjecture}[Babai~\cite{babai1978chromatic}]
    There is a constant $c$ such that every semi-minimal Cayley graph has chromatic number at most $c$.
\end{conjecture}

Garcia-Marco and Knauer~\cite{garcia2024colouring} recently constructed a semi-minimal Cayley graph with chromatic number 7.

\section*{Acknowledgements}

This work was initiated at the 11th Annual Workshop on Geometry and Graphs, which was held at Bellairs Research Institute in March 2024. We are grateful to the organisers and participants for providing a stimulating research environment.
We thank Kolja Knauer for sharing the problem and for early discussions along with Freddie Illingworth.
The work was completed when the first and third author visited the second author at the NII in Tokyo, and we would like to thank the NII for their hospitality. We also thank Amir Yasuda for his encouragement on this project.

\bibliographystyle{alpha}
\bibliography{literature}

\newcommand{\etalchar}[1]{$^{#1}$}
\begin{thebibliography}{DKK{\etalchar{+}}21}

\bibitem[AKR{\etalchar{+}}16]{alon2016coloring}
Noga Alon, Alexandr Kostochka, Benjamin Reiniger, Douglas~B West, and Xuding
  Zhu.
\newblock Coloring, sparseness and girth.
\newblock {\em Israel Journal of Mathematics}, 214(1):315--331, 2016.

\bibitem[Bab78]{babai1978chromatic}
Laszlo Babai.
\newblock Chromatic number and subgraphs of {C}ayley graphs.
\newblock Theor. {Appl}. {Graphs}, {Proc}. {Kalamazoo} 1976, {Lect}. {Notes}
  {Math}. 642, 10-22, 1978.

\bibitem[Bab95]{babai1995automorphism}
Laszlo Babai.
\newblock Automorphism groups, isomorphism, reconstruction ({C}hapter 27 of the
  {H}andbook of {C}ombinatorics).
\newblock {\em North-Holland--Elsevier}, pages 1447--1540, 1995.

\bibitem[BD23]{bucic2023explicit}
Matija Buci{\'c} and James Davies.
\newblock Explicit unit distance graphs with exponential chromatic number and
  arbitrary girth.
\newblock {\em arXiv preprint arXiv:2312.06898}, 2023.

\bibitem[BMK18]{bardestani2018chromatic}
Mohammad Bardestani and Keivan Mallahi-Karai.
\newblock On the chromatic number of structured {C}ayley graphs.
\newblock {\em Journal of Combinatorial Theory, Series A}, 160:202--216, 2018.

\bibitem[Dav21]{Davies21}
James Davies.
\newblock Box and segment intersection graphs with large girth and chromatic
  number.
\newblock {\em Advances in Combinatorics}, 2021:7, 9 pp., 2021.

\bibitem[Dav23]{davies2023chromatic}
James Davies.
\newblock Chromatic number of spacetime.
\newblock {\em arXiv preprint arXiv:2308.16885}, 2023.

\bibitem[Dav24]{davies2024odd}
James Davies.
\newblock Odd distances in colourings of the plane.
\newblock {\em Geometric and Functional Analysis}, 34(1):19--31, 2024.

\bibitem[Des47]{Descartes47}
Blanche Descartes.
\newblock A three colour problem.
\newblock {\em Eureka}, 9(21):24--25, 1947.

\bibitem[Des54]{Descartes54}
Blanche Descartes.
\newblock Solution to advanced problem no. 4526.
\newblock {\em The American Mathematical Monthly}, 61:352, 1954.

\bibitem[DKK{\etalchar{+}}21]{davies2021solution}
James Davies, Chaya Keller, Linda Kleist, Shakhar Smorodinsky, and Bartosz
  Walczak.
\newblock A solution to {R}ingel's circle problem.
\newblock {\em arXiv preprint arXiv:2112.05042}, 2021.

\bibitem[DMP23]{davies2023prime}
James Davies, Rose McCarty, and Micha{\l} Pilipczuk.
\newblock Prime and polynomial distances in colourings of the plane.
\newblock {\em arXiv preprint arXiv:2308.02483}, 2023.

\bibitem[Erd59]{erdos1959graph}
Paul Erd{\H{o}}s.
\newblock Graph theory and probability.
\newblock {\em Canadian Journal of Mathematics}, 11:34--38, 1959.

\bibitem[FM16]{farrokhi2016groups}
MDG Farrokhi and A~Mohammadian.
\newblock Groups whose all (minimal) {C}ayley graphs have a given forbidden
  structure (research on finite groups and their representations, vertex
  operator algebras, and algebraic combinatorics).
\newblock {\em RIMS Kôkyûroku}, 2003:119--127, 2016.

\bibitem[Fur77]{furstenberg1977ergodic}
Harry Furstenberg.
\newblock Ergodic behavior of diagonal measures and a theorem of
  {S}zemer{\'e}di on arithmetic progressions.
\newblock {\em Journal d’Analyse Math{\'e}matique}, 31(1):204--256, 1977.

\bibitem[GK24]{garcia2024colouring}
Ignacio {Garc{\'\i}a-Marco} and Kolja Knauer.
\newblock Coloring minimal {C}ayley graphs.
\newblock {\em arXiv preprint arXiv:2405.19543}, 2024.

\bibitem[God81]{godsil1981connectivity}
Chris Godsil.
\newblock Connectivity of minimal {C}ayley graphs.
\newblock {\em Archiv der Mathematik}, 37:473--476, 1981.

\bibitem[KM78]{kamae1978van}
Teturo Kamae and Michel {Mend{\`e}s France}.
\newblock Van der {C}orput’s difference theorem.
\newblock {\em Israel Journal of Mathematics}, 31:335--342, 1978.

\bibitem[KN99]{kostochka1999properties}
Alexandr~V Kostochka and Jaroslav Ne{\v{s}}et{\v{r}}il.
\newblock Properties of {D}escartes' construction of triangle-free graphs with
  high chromatic number.
\newblock {\em Combinatorics, Probability and Computing}, 8(5):467--472, 1999.

\bibitem[K{\v{r}}{\'\i}89]{kvrivz1989ahypergraph}
Igor K{\v{r}}{\'\i}{\v{z}}.
\newblock A hypergraph-free construction of highly chromatic graphs without
  short cycles.
\newblock {\em Combinatorica}, 9:227--229, 1989.

\bibitem[Lov68]{lovasz1968chromatic}
L{\'a}szl{\'o} Lov{\'a}sz.
\newblock On chromatic number of finite set-systems.
\newblock {\em Acta Mathematica Academiae Scientiarum Hungarica}, 19:59--67,
  1968.

\bibitem[LPS88]{lubotzky1988ramanujan}
Alexander Lubotzky, Ralph Phillips, and Peter Sarnak.
\newblock Ramanujan graphs.
\newblock {\em Combinatorica}, 8(3):261--277, 1988.

\bibitem[LZ01]{li2001isomorphisms}
Cai~Heng Li and Sanming Zhou.
\newblock On isomorphisms of minimal {C}ayley graphs and digraphs.
\newblock {\em Graphs and Combinatorics}, 17:307--314, 2001.

\bibitem[M{\v{S}}20]{miklavivc2020minimal}
{\v{S}}tefko Miklavi{\v{c}} and Primo{\v{z}} {\v{S}}parl.
\newblock On minimal distance-regular {C}ayley graphs of generalized dihedral
  groups.
\newblock {\em The Electronic Journal of Combinatorics}, pages P4--33, 2020.

\bibitem[Nic24]{nica2024independence}
Bogdan Nica.
\newblock On the independence number of regular graphs of matrix rings.
\newblock {\em Linear Algebra and its Applications}, 681:89--96, 2024.

\bibitem[NR79]{nevsetvril1979short}
Jaroslav Ne{\v{s}}et{\v{r}}il and Vojt{\v{e}}ch R{\"o}dl.
\newblock A short proof of the existence of highly chromatic hypergraphs
  without short cycles.
\newblock {\em Journal of Combinatorial Theory, Series B}, 27(2):225--227,
  1979.

\bibitem[PV88]{PV88}
Hans~J\"urgen Pr\"omel and Bernd Voigt.
\newblock A sparse {G}raham-{R}othschild theorem.
\newblock {\em Transactions of the American Mathematical Society},
  309(1):113--137, 1988.

\bibitem[PV90]{PV90}
Hans~J\"urgen Pr\"omel and Bernd Voigt.
\newblock A sparse {G}allai-{W}itt theorem.
\newblock In Rainer Bodendiek and Rudolf Henn, editors, {\em Topics in
  Combinatorics and Graph Theory}, pages 747--755. Physica-Verlag Heidelberg,
  1990.

\bibitem[Rad45]{Rado45}
Richard Rado.
\newblock Note on combinatorial analysis.
\newblock {\em Proceedings of the London Mathematical Society}, 48:122--160,
  1945.

\bibitem[Rin59]{Ringel59}
Gerhard Ringel.
\newblock {\em F\"{a}rbungsprobleme auf Fl\"{a}chen und Graphen}, volume~2 of
  {\em Mathematische Monographien}.
\newblock VEB Deutscher Verlag der Wissenschaften, 1959.

\bibitem[S{\'a}r78]{sarkozy1978difference}
Andr{\'a}s S{\'a}rk{\"o}zy.
\newblock On difference sets of sequences of integers. {I}.
\newblock {\em Acta Mathematica Academiae Scientiarum Hungarica}, 31:125--149,
  1978.

\bibitem[Spe83]{spencer1983s}
Joel Spencer.
\newblock What’s not inside a {C}ayley graph.
\newblock {\em Combinatorica}, 3:239--241, 1983.

\bibitem[SS09]{slupik2009minimal}
Anna~J Slupik and Vitaly~I Sushchansky.
\newblock Minimal generating sets and {C}ayley graphs of {S}ylow p-subgroups of
  finite symmetric groups.
\newblock {\em Algebra and Discrete Mathematics}, 2009.

\bibitem[Tom15]{tomon2015chromatic}
Istv{\'a}n Tomon.
\newblock On the chromatic number of regular graphs of matrix algebras.
\newblock {\em Linear Algebra and its Applications}, 475:154--162, 2015.

\end{thebibliography}

\end{document}